\documentclass[draft, 12pt]{amsart}
\usepackage[english]{babel}

\usepackage{amssymb} 
\usepackage{amsmath}
\usepackage[all,cmtip]{xy}
\usepackage[pagewise]{lineno}

\usepackage{color}
\newtheorem{thm}{Theorem}

\newtheorem{prop}[thm]{Proposition}

\newtheorem{lem}[thm]{Lemma}

\usepackage{ifdraft, blindtext, verbatim}
\ifdraft{%
    \usepackage{todonotes}}{ 
    \usepackage[disable]{todonotes} 
}

\begin{document}
\title[Weinstein's Morphism]
{Generalization of Weinstein's Morphism}

\author{Andr\'es Pedroza}  
\address{Facultad de Ciencias\\
           Universidad de Colima\\
           Bernal D\'{\i}az del Castillo No. 340\\
           Colima, Col., Mexico 28045}
\email{andres\_pedroza@ucol.mx}

\begin{abstract}  
We introduce a generalization of Weinstein's morphism, defined on  
\(\pi_{2k-1}(\mathrm{Ham}(M, \omega))\) for \(1 < k \leq n\), where \((M, \omega)\) is a \(2n\)-dimensional 
symplectic manifold.  
Using this morphism, we show that for \(n > 1\) and \(1 < k \leq n\), the homotopy groups  
\[
\pi_{2k-1}\!\big(\mathrm{Ham}(\mathbb{C}P^n, \omega_\textup{FS})\big)
\quad \text{and} \quad
\pi_{2k-1}\!\big(\mathrm{Ham}(\widetilde{\mathbb{C}P}{}^n, \widetilde{\omega}_\rho)\big)
\]
are nontrivial.  
Here, \((\widetilde{\mathbb{C}P}{}^n, \widetilde{\omega}_\rho)\) denotes the symplectic one-point 
blow-up of \((\mathbb{C}P^n, \omega_\textup{FS})\) of weight~\(\rho\).   Further applications of the morphism
include the group $\pi_{2k-1}\!\big(\mathrm{Ham} (\mathbb{C}P^n\times M, \omega_\textup{FS}\oplus\omega)\big)$.
\end{abstract}

\keywords{Weinstein's morphism, Hamiltonian diffeomorphism}


\subjclass{Primary: 57S05, 53D35 Secondary:  57R17, 53D05}

\maketitle

\section{Introduction}


Let $(M,\omega)$ be a closed symplectic manifold of dimension $2n$.
For $1\leq k \leq n $ define the group
 $\mathcal{P}_{2k}(M,\omega)$   as the image of the pairing 
$
\langle\omega^k,\cdot \rangle\colon H_{2k}(M;\mathbb{Z})
\to \mathbb{R}.
$ 
Notice that for  $k=1$,   $\mathcal{P}_{2}(M,\omega)$ is the well-known period group
of  the symplectic manifold $(M,\omega)$. 
In \cite{weinstein-coho}, A. Weinstein defined a morphism
\begin{eqnarray*}
\mathcal{A}\colon
 \pi_{1} ( \textup{Ham} (M,\omega)) \to \mathbb{R}/\mathcal{P}_{2}
(M,\omega)
\end{eqnarray*}
based on the action functional of a Hamiltonian function $H:M\times [0,1]\to\mathbb{R}$
associated with the loop of Hamiltonian diffeomorphisms. 
Inspired by Weinstein's morphism, we define in this article a group morphism  
\[
\mathcal{A} \colon \pi_{2k-1}(\textup{Ham}(M, \omega)) \to \mathbb{R} / \mathcal{P}_{2k}(M, \omega)
\]  
for each \( k \in \{2, \ldots, n\} \).
The value of \( \mathcal{A} \) on an element \([\psi_{\underline{t}}] \in \pi_{2k-1}
(\textup{Ham}(M, \omega))\) can be broadly interpreted as the average, 
over the symplectic manifold \((M, \omega)\), of the \(\omega^k\)-volume of the 
\(2k\)-balls bounded by the topological \((2k-1)\)-spheres  
$
\{ \psi_{\underline{t}}(p) : \underline{t} \in S^{2k-1} \}
$  
as the base point \(p \in M\) varies.

We proceed to compute the value of \( \mathcal{A} \) on a specific element of the group  
\(\pi_{2k-1}(\textup{Ham}(\mathbb{C}P^n, \omega_{\textup{FS}}))\), for \(1 \leq k \leq n\).  
This element, which we denoted by \([\{ \psi_{A_{\underline{t}}} \}]_{{\underline{t}} \in S^{2k-1}}\),  
is induced by a family of unitary matrices \( A_{\underline{t}} \in U(n+1) \) representing a 
generator of the homotopy group \(\pi_{2k-1}(U(n+1))\).
The computation relies on the existence of a symplectic embedding  
$
(B^{2n}(1), \omega_0) \hookrightarrow (\mathbb{C}P^n, \omega_{\textup{FS}})
$  
of the unit open \(2n\)-ball whose image is dense in \(\mathbb{C}P^n\).  
A detailed description of this embedding is provided in the appendix of \cite{mcduffpol-packing}.  
Using this embedding, we carry out the explicit computation of  
$
\mathcal{A}\left( [\{ \psi_{A_{\underline{t}}} \}]_{{\underline{t}} \in S^{2k-1}} \right)
$
in Proposition \ref{p:2n}. This yields an alternative and constructive proof of the 
nontriviality of the homotopy groups  \(\pi_{2k-1}(\textup{Ham}(\mathbb{C}P^n, 
\omega_{\textup{FS}}))\), for all \(1 \leq k \leq n\).

\begin{thm}
\label{t:main}
Let \( (\mathbb{C}P^n, \omega_{\textup{FS}}) \) denote complex projective space equipped 
with the Fubini–Study symplectic form. Then, for every integer \( k \in \{1, \ldots, n\} \), the 
homotopy group  
\[
\pi_{2k-1}(\textup{Ham}(\mathbb{C}P^n, \omega_{\textup{FS}}))
\]  
is nontrivial.
\end{thm}

This result was originally established by A. G. Reznikov in more general terms in 
\cite[Theorem 1.4]{reznikov}, through 
the construction of odd-degree characteristic classes in the cohomology ring of \(\textup{Ham}(M, \omega)\). 
For the cases \(k = 1\) and \(k = 2\), alternative proofs using different techniques are also available.  
When \(k = 1\), a specific Hamiltonian circle action yields a nontrivial element in 
\(\pi_{1}(\textup{Ham}(\mathbb{C}P^n, \omega_{\textup{FS}}))\). This can be shown 
using either Weinstein’s morphism \cite{weinstein-coho} or Seidel’s morphism 
\cite{seidel-pi1of}.  
For \(k = 2\), it is known that \(\pi_3(\textup{Ham}(\mathbb{C}P^n, \omega_{\textup{FS}}))\) 
contains a subgroup isomorphic to \(\mathbb{Z}\) for \(n \geq 2\). In \cite[Theorem 1.1]{KedraMcDuff}, 
J. K\c{e}dra and D. McDuff showed that if a symplectic manifold \((M, \omega)\) admits a
 nontrivial Hamiltonian circle action that is contractible in \(\textup{Ham}(M, \omega)\), then
  \(\pi_3(\textup{Ham}(M, \omega))\) contains an element of infinite order. This result applies, 
  in particular, to \((\mathbb{C}P^n, \omega_{\textup{FS}})\).  
Moreover, M. Gromov \cite{gromov-psudo} proved that \(\textup{Ham}(\mathbb{C}P^2, 
\omega_{\textup{FS}})\) has the homotopy type of \(PU(3)\), from which it follows that  
$
\pi_3(\textup{Ham}(\mathbb{C}P^2, \omega_{\textup{FS}})) \simeq \mathbb{Z}.
$

In addition, the nontrivial element  
\([\{ \psi_{A_{\underline{t}}} \}]_{{\underline{t}} \in S^{2k-1}}\)  
in  
\(\pi_{2k-1}(\textup{Ham}(\mathbb{C}P^n, \omega_{\textup{FS}}))\), for \(1 \leq k \leq n\),  
admits a lift to the symplectic one-point  blow-up  
\((\widetilde{\mathbb{C}P}\,^n, \widetilde{\omega}_\rho)\) of weight \(\rho\).  
By computing the value of  \(\mathcal{A}\) on the lifted element  
\([\{ \widetilde{\psi}_{A_{\underline{t}}} \}]_{{\underline{t}} \in S^{2k-1}}\),
in  Proposition \ref{p:ainblow},   
we show that it represents an element of infinite order in
\(\pi_{2k-1}(\textup{Ham}(\widetilde{\mathbb{C}P}\,^n, \widetilde{\omega}_\rho))\).
In contrast with Theorem~\ref{t:main}, the next statement appears to be new.

\begin{thm}  
\label{t:realmain}  
Let \( n > 1 \), and let \( \rho \in (0,1) \) be a transcendental number. Then, for every \( k \in \{1, \ldots, n\} \), the homotopy group  
\[
\pi_{2k-1}\big(\textup{Ham}(\widetilde{\mathbb{C}P}\,^n, \widetilde{\omega}_\rho)\big)
\]  
contains an element of infinite order.  
\end{thm}

The approach to proving this result is similar to that used in \cite{pea-rankham} , where Weinstein's morphism
 is employed to show that \( \pi_{1}\big(\textup{Ham}(\widetilde{M}, \widetilde{\omega}_\rho)\big) \) 
 has positive rank. The key idea is to analyze the relation expressing the action functional 
 \( \mathcal{A}[\widetilde{\psi}_{\underline{t}}] \) in terms of \( \mathcal{A}[\psi_{\underline{t}}] \). 
 It turns out that the vanishing of \( \mathcal{A}[\widetilde{\psi}_{\underline{t}}] \) is equivalent to
the vanishing of a polynomial in \( \rho \) with rational coefficients.

\newpage

Although the statement concerning 
$
\pi_{2k-1}\!\big(\mathrm{Ham}(\mathbb{C}P^n, \omega_{\mathrm{FS}})\big)
$
in Theorem~\ref{t:main} is not new, the use of Weinstein’s morphism in the computation of  
$
\mathcal{A}\!\left([\{\psi_{A_{\underline{t}}}\}]\right)
$
allows us to arrive at an analogous statement for other symplectic manifolds.
Namely, in Proposition~\ref{p:product} we analyzed how to compute Weinstein’s 
morphism in the case of Cartesian products of symplectic manifolds, that is, on 
\(\pi_{2k-1}(\mathrm{Ham}(M\times N,\omega\oplus\eta))\) in terms of the 
Weinstein morphisms on 
\(\pi_{2k-1}(\mathrm{Ham}(M,\omega))\) and 
\(\pi_{2k-1}(\mathrm{Ham}(N,\eta))\).  
This formulation leads to the following result.

\begin{thm}
\label{t:prod}
Let $(M,\omega)$ be a closed symplectic manifold of dimension $2m$, and let 
$n,k\in \mathbb{N}$ with $1\leq k\leq \min\{m,n\}$. 
Assume that $\pi_{2k-1}(M)=0$ and that
 $\mathcal{P}_{2j}(M,\omega)\subset \mathbb{Q}$ for all $1\leq j\leq k$.
Then
\[
\pi_{2k-1}\!\left(\mathrm{Ham}(\mathbb{C}P^n\times M, \, 
\omega_{\mathrm{FS}}\oplus \omega)\right)
\]
is nontrivial.
\end{thm}

The previous result admits an analogous formulation obtained by replacing 
\((\mathbb{C}P^n, \omega_{\mathrm{FS}})\) with 
\((\widetilde{\mathbb{C}P}{}^n, \widetilde{\omega}_\rho)\). 
The proof follows the same line of argument as before; 
for this reason, and to avoid unnecessary repetition, 
we omit its statement.

\section{Definition of the morphism $\mathcal{A}$}

Fix an integer \( 1 \leq k \leq n \), and suppose the symplectic manifold \( (M, \omega) \) 
satisfies \( \pi_{2k-1}(M) = 0 \). Let \( D^{2k} \) denote the closed \( 2k \)-dimensional ball. 
Then, for any smooth map \( \gamma \in \Omega_{2k-1}(M) \), there exists a smooth extension  
$
u_\gamma \colon D^{2k} \to M
$
such that \( u_\gamma|_{\partial D^{2k}} = \gamma \).  
Define a map  
$
A \colon \Omega_{2k-1}(M) \to \mathbb{R} / \mathcal{P}_{2k}(M, \omega)
$
by  
\begin{eqnarray}
\label{e:defofA}
A(\gamma):=\int_{D^{2k}} u_\gamma^*(\omega^k)
\end{eqnarray}
  where \( u_\gamma \) is any extension of \( \gamma \) as above.
 Since   $A$  takes values in $\mathbb{R}/
\mathcal{P}_{2k}(M,\omega)$, 
it follows that \( A(\gamma) \)  is
independent of the choice of extension \( u_\gamma \).

Let \(\{ \psi_{\underline{t}}\}  \in \Omega_{2k-1}(\textup{Ham}(M,\omega)) \) and fix a point 
\( p \in M \). Define the trace of \( p \) under \( \{ \psi_{\underline{t}}\} \) by  
\[
\gamma(\{\psi_{\underline{t}}\}_{{\underline{t}}\in S^{2k-1}}, p) := \{\psi_{\underline{t}}(p) \mid \underline{t} \in S^{2k-1}\} \subset M.
\]
This trace defines a topological \( (2k-1) \)-sphere in \( M \), and thus determines an element of \( \Omega_{2k-1}(M) \).  
We now define the map  
$
\mathcal{A} \colon \Omega_{2k-1}(\textup{Ham}(M,\omega)) \to \mathbb{R} / \mathcal{P}_{2k}(M,\omega)
$  
by averaging the value of \( A \)  over \( M \) with respect to the symplectic volume form. Specifically,  
\begin{equation}
\label{e:accion}
\mathcal{A}(\psi_{\underline{t}}) := \frac{1}{\textup{Vol}(M, \omega^n/n!)} 
\int_M A\left(\gamma(\{\psi_{\underline{t}}\}_{\underline{t} \in S^{2k-1}}, p)\right) \frac{\omega^n}{n!}.
\end{equation}

This definition yields a well-defined function on 
$\Omega_{2k-1}(\textup{Ham}(M,\omega))$. 
Next, we show that the map \( \mathcal{A} \) descends to the group 
\( \pi_{2k-1}(\textup{Ham}(M,\omega)) \).

\begin{lem}
If \( \psi^{(0)}_{\underline{t}}, \psi^{(1)}_{\underline{t}} \in \Omega_{2k-1}(\textup{Ham}(M,\omega)) \) 
are homotopic elements and
$
[\psi^{(0)}_{\underline{t}}] = [\psi^{(1)}_{\underline{t}}] \in \pi_{2k-1}(\textup{Ham}(M,\omega)),
$  
then
\[
\mathcal{A}(\psi^{(0)}_{\underline{t}}) = \mathcal{A}(\psi^{(1)}_{\underline{t}}).
\]
\end{lem}
\begin{proof}
Let  
$
\mathbf{\Psi} \colon [0,1] \times S^{2k-1} \to \textup{Ham}(M, \omega)
$  
be a smooth homotopy between \( \psi^{(0)}_{\underline{t}} \) and \( \psi^{(1)}_{\underline{t}} \). 
Fix a point \( p \in M \), and denote by  
$
u_0, u_1 \colon D^{2k} \to M
$  
the capping maps of the \( (2k-1) \)-spheres \( \gamma(\{\psi^{(0)}_{\underline{t}}\}, p) \)
 and \( \gamma(\{\psi^{(1)}_{\underline{t}}\}, p) \), respectively.

Evaluating the homotopy \( \mathbf{\Psi} \) at \( p \), we obtain a smooth map  
$
\mathbf{\Psi}_p \colon [0,1] \times S^{2k-1} \to M
$  
that gives a homotopy between the spheres 
\( \gamma(\{\psi^{(0)}_{\underline{t}}\}, p) \)
 and \( \gamma(\{\psi^{(1)}_{\underline{t}}\}, p) \)
in \( M \). As a result, we may glue the maps \( u_0 \), \( \mathbf{\Psi}_p \), 
and \( \overline{u_1} \) (the map \( u_1 \) with reversed orientation on the domain) to 
obtain a smooth map from the \( 2k \)-sphere into \( M \).
Thus, the sum of the corresponding integrals satisfies
$$
\int_{D^{2k}} u_0^*(\omega^k) + \int_{[0,1] \times S^{2k-1}} \mathbf{\Psi}_p^*(\omega^k) + \int_{D^{2k}} \overline{u_1}^*(\omega^k) = 0
$$
in \( \mathbb{R}/\mathcal{P}_{2k}(M,\omega) \).

Now, observe that the second integral vanishes,
since \( \mathbf{\Psi}_p \) is a trace of a family of Hamiltonian diffeomorphisms.
It follows that
\[
\int_{D^{2k}} u_0^*(\omega^k) = \int_{D^{2k}} u_1^*(\omega^k)
\]
modulo \( \mathcal{P}_{2k}(M,\omega) \), and thus
$
A(\gamma(\{\psi^{(0)}_{\underline{t}}\}, p)) = A(\gamma(\{\psi^{(1)}_{\underline{t}}\}, p)).
$
Since this equality holds for all \( p \in M \), we conclude that
$
\mathcal{A}(\psi^{(0)}_{\underline{t}}) = \mathcal{A}(\psi^{(1)}_{\underline{t}}).
$
\end{proof}

Thus, the map  
$
\psi_{\underline{t}} \mapsto \mathcal{A}(\psi_{\underline{t}})
$  
descends to a well-defined map on homotopy classes,
\[
\mathcal{A}^{M}_{2k-1} \colon \pi_{2k-1}(\textup{Ham}(M, \omega)) \to \mathbb{R}/\mathcal{P}_{2k}(M, \omega),
\]
Moreover, by the definition of \( A \) in Equation~\eqref{e:defofA}, it satisfies the 
identities
\[
A(-\gamma) = -A(\gamma), \qquad 
A(\gamma_1 + \gamma_2) = A(\gamma_1) + A(\gamma_2),
\]
for all \( \gamma, \gamma_1, \gamma_2 \in \Omega_{2k-1}(M) \). 
These properties imply that $\mathcal{A}^{M}_{2k-1}$ is a group homomorphism. 
In order to avoid cumbersome notation, we will simply write \( \mathcal{A} \)
without explicitly indicating the underlying symplectic manifold or the degree of 
the homotopy group in which it is defined. Whenever there is a risk of confusion, 
the full notation will be reinstated.

We are now in a position to state a generalization of Weinstein’s morphism. Let \( (M, \omega) \) 
be a closed symplectic manifold of dimension \( 2n \), and fix \( k \in \{2, \ldots, n\} \). Assume 
that the homotopy group \( \pi_{2k-1}(M) \) is trivial. We define the {\em generalized Weinstein morphism}
 as the map
\[
\mathcal{A} \colon \pi_{2k-1}(\textup{Ham}(M, \omega)) \to \mathbb{R}/\mathcal{P}_{2k}(M, \omega),
\]
given by
\[
\mathcal{A}(\psi_{\underline{t}}) := \frac{1}{\textup{Vol}(M, \omega^n/n!)} \int_M 
A\big(\gamma(\{\psi_{\underline{t}}\}_{\underline{t} \in S^{2k-1}}, p)\big) \, \frac{\omega^n}{n!},
\]
for any \( \psi_{\underline{t}} \in \pi_{2k-1}(\textup{Ham}(M, \omega)) \).  

This construction generalizes the classical Weinstein morphism, which corresponds to 
the case \( k = 1 \). In higher degrees, the morphism \( \mathcal{A} \) captures obstructions 
to the triviality of higher homotopy classes of Hamiltonian diffeomorphisms, encoded via 
their averaged action over the symplectic manifold.

In \cite{weinstein-coho}, an alternative formulation of the classical Weinstein morphism  
\[
\mathcal{A}\colon \pi_{1}(\textup{Ham}(M, \omega)) \to \mathbb{R}/\mathcal{P}_{2}(M, \omega)
\]
is presented in addition to the one described here. It is given by
\[
\mathcal{A}(\psi^H) = \int_{D^2} u_{x_0}^*(\omega) + \int_0^1 H_t(x_0) \, dt,
\]
where \( H_t \) is a normalized Hamiltonian function generating the loop \( \psi^H \), \( x_0 \in M \) is a 
fixed base point, and \( u_{x_0} \colon D^2 \to M \) is a smooth map capping the 
loop \( \{ \psi_t^H(x_0) \} \). 
It is worth emphasizing that this alternative description does not admit a natural generalization 
for \( k > 1 \).

Additionally, it is worth noting that the definition of the morphism \( \mathcal{A} \) for \( k > 1 \) 
does not involve any explicit reference to Hamiltonian functions. This omission is, in part, due 
to the fact that the inclusion  
\[
\textup{Ham}(M, \omega) \hookrightarrow \textup{Symp}_0(M, \omega)
\]  
induces an isomorphism on the homotopy groups \( \pi_k \) for \( k > 1 \). 
Consequently, the use of explicit Hamiltonian functions is not necessary in defining \( \mathcal{A} \)
 in these higher-degree settings.

\subsection{Cartesian Products}

We now turn to the analysis of Weinstein’s morphism under symplectic Cartesian products. 
Such an analysis broadens the range of potential applications of the morphism, 
particularly in light of explicit computations in cases such as
\[
\pi_{2k-1}\!\big(\mathrm{Ham}(\mathbb{C}P^n, \omega_{\mathrm{FS}})\big)
\quad \text{and} \quad
\pi_{2k-1}\!\big(\mathrm{Ham}(\widetilde{\mathbb{C}P}{}^n, \widetilde{\omega}_\rho)\big).
\]
Recall that if 
\(\psi\) and \(\phi\) are Hamiltonian diffeomorphisms of \((M,\omega)\) and \((N,\eta)\), respectively, 
then their product \(\psi \times \phi\) defines a Hamiltonian diffeomorphism of the product symplectic manifold 
\((M \times N, \omega \oplus \eta)\).

\begin{prop}
\label{p:product}
Let $(M,\omega)$ and $(N,\eta)$ be closed symplectic manifolds, and let 
$k\in \mathbb{Z}$ be an integer such that 
$
1\leq k \leq \min\{m,n\},
$
where $2m=\dim M$ and $2n=\dim N$. Assume that the groups
$\pi_{2k-1}(M)$ and $\pi_{2k-1}(N)$ are trivial
If $[\psi_{\underline{t}}]\in \pi_{2k-1}(\mathrm{Ham}(M,\omega))$ and 
$[\phi_{\underline{t}}]\in \pi_{2k-1}(\mathrm{Ham}(N,\eta))$, then for 
$[\psi_{\underline{t}} \times \phi_{\underline{t}}]\in \pi_{2k-1}(\mathrm{Ham}(M\times N,\omega\oplus\eta))$
we have
\[
\mathcal{A}^{M\times N}[\psi_{\underline{t}} \times \phi_{\underline{t}}]
=\big[\,\mathcal{A}^{M}[\psi_{\underline{t}}]+\mathcal{A}^{N}[\phi_{\underline{t}}]\,\big]
\in \mathbb{R}/\mathcal{P}_{2k}(M\times N,\omega\oplus \eta).
\]
\end{prop}

\begin{proof}
Fix $p \in M$ and $q \in N$. 
It follows directly from the definition of $A$ that for 
$[\psi_{\underline{t}}] \in \pi_{2k-1}(\mathrm{Ham}(M,\omega))$,
\[
A^{M\times N}\!\left(
\gamma\!\big(\{\psi_{\underline{t}} \times 1_N\}_{\underline{t} \in S^{2k-1}}, (p,q)\big)
\right)
= 
\left[\,
A^{M}\!\left(
\gamma\!\big(\{\psi_{\underline{t}}\}_{\underline{t} \in S^{2k-1}}, p\big)
\right)
\,\right],
\]
and therefore,
\begin{equation*}
\mathcal{A}^{M\times N}[\psi_{\underline{t}} \times 1_N] 
= 
\big[\,\mathcal{A}^{M}[\psi_{\underline{t}}]\,\big]
\;\in\; 
\mathbb{R}/\mathcal{P}_{2k}(M\times N,\omega \oplus \eta).
\end{equation*}

In $\pi_{2k-1}(\mathrm{Ham}(M\times N,\omega\oplus\eta))$ there is the 
decomposition
\[
[\psi_{\underline{t}} \times \phi_{\underline{t}}]
= [\psi_{\underline{t}} \times 1_N]\;*\;[1_M \times \phi_{\underline{t}}].
\]
Thus, from the above computation it follows that
\begin{align*}
\mathcal{A}^{M\times N}[\psi_{\underline{t}} \times \phi_{\underline{t}}]
&= \mathcal{A}^{M\times N}[\psi_{\underline{t}} \times 1_N] \;+\;
   \mathcal{A}^{M\times N}[1_M \times \phi_{\underline{t}}] \\
&= \big[\,\mathcal{A}^{M}[\psi_{\underline{t}}] 
   + \mathcal{A}^{N}[\phi_{\underline{t}}]\,\big],
\end{align*}
as claimed.
\end{proof}

Note that for closed symplectic manifolds $(M,\omega)$ and $(N,\eta)$ one has
\[
\mathcal{P}_{2}(M,\omega)\;+\;\mathcal{P}_{2}(N,\eta)
\;=\;
\mathcal{P}_{2}(M\times N,\omega\oplus\eta),
\qquad (k=1),
\]
while for $k>1$ we can only assert that
\[
\mathcal{P}_{2k}(M,\omega)\;+\;\mathcal{P}_{2k}(N,\eta)
\;\subset\;
\mathcal{P}_{2k}(M\times N,\omega\oplus\eta).
\]

We will use this result at the end of Section~\ref{hamdiff}, where
we find the value of the Weinstein morphism on
$\pi_{2k-1}\!\big(\mathrm{Ham}(\mathbb{C}P^n, \omega_{\mathrm{FS}})\big)$,
in order to provide the proof of 
Theorem~\ref{t:prod} stated in the Introduction.

\section{Application: $\pi_{2k-1}(\textup{Ham} (\mathbb{C}P^n,\omega_\textup{FS}) )$  for
$1\leq k\leq n$}
\label{hamdiff}


Recall that each matrix \( A \in U(n+1) \) induces a Hamiltonian diffeomorphism  
$
\psi_A \colon (\mathbb{C}P^n, \omega_{\textup{FS}}) \to (\mathbb{C}P^n, \omega_{\textup{FS}})
$  
defined by  
\[
\psi_A([w]) = [Aw],
\]  
where \( Aw \) denotes the standard matrix. This association 
defines a group morphism  
$
U(n+1) \to \textup{Ham}(\mathbb{C}P^n, \omega_{\textup{FS}}).
$  
The aim of this section is to study the image of the induced map on homotopy groups  
\[
\pi_{2k-1}(U(n+1)) \to \pi_{2k-1}(\textup{Ham}(\mathbb{C}P^n, \omega_{\textup{FS}})),
\]  
for integers \( k \in \{1, \ldots, n\} \), under the generalized Weinstein morphism \( \mathcal{A} \) 
introduced in the previous section. 
 Note that for these values of \( k \), the group 
\( \pi_{2k-1}(U(n+1)) \) is isomorphic to \( \mathbb{Z} \), and hence the analysis 
involves understanding the image of a generator.

To that end, consider the complex projective space \( (\mathbb{C}P^n, \omega_{\textup{FS}}) \), where the symplectic form is normalized so that the symplectic area of a complex projective line equals \( \pi \). With this normalization, there exists a symplectic embedding  
$
j \colon (B^{2n}(1), \omega_0) \hookrightarrow (\mathbb{C}P^n, \omega_{\textup{FS}})
$  
defined by  
\begin{eqnarray}
\label{e:simpemn}
(z_1, \ldots, z_n) \mapsto \left[ z_1 : \cdots : z_n : \sqrt{1 - \sum_{j=1}^n |z_j|^2} \right],
\end{eqnarray}
where the domain is the standard unit open ball in \( \mathbb{C}^n \) equipped with the standard symplectic form \( \omega_0 \). For a detailed proof of this symplectic embedding, we refer the reader to the Appendix of \cite{mcduffpol-packing}.
The image of this embedding is dense in \( \mathbb{C}P^n \), and its complement is the 
hyperplane  
$
\mathbb{C}P^{n-1} = \left\{ [w_1 : \cdots : w_n : 1] \right\}.
$  
Indeed, given a point \( [w] = [w_1 : \cdots : w_{n+1}] \in \mathbb{C}P^n \setminus \mathbb{C}P^{n-1} \), one can verify that  
\[
\left( 1 + \sum_{k=1}^n \left| \frac{w_k}{w_{n+1}} \right|^2 \right)^{-1/2} \left( \frac{w_1}{w_{n+1}}, \ldots, \frac{w_n}{w_{n+1}} \right) \in B^{2n}(1)
\]  
is mapped to \( [w] \) via the symplectic embedding \( j \).

Now consider the subgroup of unitary matrices of the form  
\begin{eqnarray}
\label{e:unit}
A = \begin{pmatrix}
A_1 & 0 \\
0 & 1
\end{pmatrix} \in U(n+1),
\end{eqnarray}
where \( A_1 \in U(n) \). Such matrices act trivially on the last coordinate and hence preserve the image of the embedding \( j \), that is,  
$
\psi_A \left( j(B^{2n}(1)) \right) = j(B^{2n}(1)).
$  
This leads to an inclusion homomorphism  
$
\iota \colon U(n) \hookrightarrow U(n+1)
$  
induced by the block structure in \eqref{e:unit}. Since the standard action of \( U(n) \) on \( B^{2n}(1) \subset \mathbb{C}^n \) preserves the Euclidean norm, it follows that  
\[
\psi_{\iota(A_1)} \circ j(z) = j(A_1 z), \quad \text{for all } z \in B^{2n}(1),\ A_1 \in U(n),
\]  
i.e., the embedding \( j \) is \( U(n) \)-equivariant with respect to the inclusion \( \iota \). In what follows, we will often omit explicit mention of \( \iota \) and simply treat elements of \( U(n) \) as acting on \( \mathbb{C}P^n \) via their inclusion in \( U(n+1) \).

Recall that the homotopy long exact sequence associated to the fibration  
$
U(n) \longrightarrow U(n+1) \longrightarrow S^{2n+1}
$  
implies that the inclusion map \( U(n) \hookrightarrow U(n+1) \), as described in \eqref{e:unit}, induces an isomorphism  
$
\pi_j(U(n)) \to \pi_j(U(n+1))
$  
for all \( 1 \leq j \leq 2n \). 
By successively applying this 
observation, we have the following 
description of unitary matrices that can represent a
generator of $\pi _{2k-1}(U(n+1))$.

\begin{lem}
\label{l:matrix}
Let \(1 \leq k \leq n\). Then a generator of the group \(\pi_{2k-1}(U(n+1))\) can be represented by 
unitary matrices of the form  
\begin{eqnarray}
\label{e:mathk}
\begin{pmatrix}
A & 0 \\
0 & I_{n-k+1}
\end{pmatrix} \in U(n+1),
\end{eqnarray}
where \(A \in U(k)\) and \(I_{n-k+1}\) denotes the \((n-k+1) \times (n-k+1)\) identity matrix.
\end{lem}

Fix \(k\) such that \(1 \leq k \leq n\), and consider a generator 
\(\{ A_{\underline{t}} \}_{{\underline{t}} \in S^{2k-1}}\) of 
 the group
\(\pi_{2k-1}(U(n+1))\), where each matrix \(A_{\underline{t}}\) is of the
 form described in Lemma \ref{l:matrix}. As before, let \(A_{{\underline{t}},1} 
\in U(n)\) denote the matrix obtained by omitting the last row and column
 of \(A_{\underline{t}}\).  
Now, fix a generic point \({\bf z} = (z_1, \ldots, z_n) \in B^{2n}(1)
\). The trace  
$
\{ A_{{\underline{t}},1} {\bf z} \mid {\underline{t}} \in S^{2k-1} \} \subset B^{2n}(1)
$  
under the family \(\{ A_{{\underline{t}},1} \}\) describes a \((2k-1)\)-sphere centered at  
$
(0, \ldots, 0, z_{k+1}, \ldots, z_n),
$  
with radius \(\sqrt{|z_1|^2 + \cdots + |z_k|^2}\), and contained in the affine complex \(k\)-plane  
$
\{ (u_1, \ldots, u_k, z_{k+1}, \ldots, z_n) \mid u_j \in \mathbb{C} \}.
$  
In particular, when \(z_1 = \cdots = z_k = 0\), the trace \(\{ A_{{\underline{t}},1} {\bf z} \mid {\underline{t}} \in S^{2k-1} \}\) consists of a single point.

We claim that the element  
$
[\psi_{A_{\underline{t}}}] \in \pi_{2k-1}(\textup{Ham}(\mathbb{C}P^n,\omega_\textup{FS}))
$  
is nontrivial.  
Note that \([\mathbb{C}P^k]\) generates the homology group \(H_{2k}(\mathbb{C}P^n;\mathbb{Z})\), and the symplectic volume satisfies  
$
\langle \omega_\textup{FS}^k, \mathbb{C}P^k \rangle = {\pi^k}/{k!}.
$  
Hence, the group of periods is given by  
$
\mathcal{P}_{2k}(\mathbb{C}P^n,\omega_\textup{FS}) = \left\langle {\pi^k}/{k!} \right\rangle.
$  
Since \(\pi_{2k-1}(\mathbb{C}P^n) = 0\), the generalized Weinstein morphism  
\[
\mathcal{A} \colon \pi_{2k-1}(\textup{Ham}(\mathbb{C}P^n,\omega_\textup{FS})) \to \mathbb{R}
 / \left\langle {\pi^k}/{k!} \right\rangle
\]  
is well-defined.
We now compute \(\mathcal{A}([\psi_{A_{\underline{t}}}])\).  
By Lemma \ref{l:matrix}, each diffeomorphism in the family \(\{\psi_{A_{\underline{t}}}\}\) has the form  
\[
\psi_{A_{\underline{t}}}[w_1:\cdots:w_{n+1}] = [A_{{\underline{t}},1}(w_1, \ldots, w_k) : w_{k+1} : \cdots : w_{n+1}],
\]  
for some \(A_{{\underline{t}},1} \in U(k)\). In particular, the image of the embedding \(j(B^{2n}(1)) \subset \mathbb{C}P^n\) is invariant under each diffeomorphism \(\psi_{A_{\underline{t}}}\), and the restriction of the action to the embedded ball corresponds to the standard \(U(k)\)-action:
\begin{eqnarray}
\label{e:definc}
A_{\underline{t}}(z_1, \ldots, z_n) = (A_{{\underline{t}},1}(z_1, \ldots, z_k), z_{k+1}, \ldots, z_n). 
\end{eqnarray}
As observed earlier, for a generic point \({\bf z} \in B^{2n}(1)\), the set \(\{ A_{{\underline{t}},1} {\bf z} \mid {\underline{t}} \in S^{2k-1} \}\) traces out the boundary of a \(2k\)-dimensional ball of radius \(\sqrt{|z_1|^2 + \cdots + |z_k|^2}\), fully contained in \(B^{2n}(1)\).

As before, for \([{\bf w}] \in \mathbb{C}P^n\), we denote by  
$
\gamma(\{\psi_{A_{\underline{t}}}\}, [{\bf w}])
$  
the trace of \([{\bf w}]\) under the family of Hamiltonian diffeomorphisms \(\{\psi_{A_{\underline{t}}}\}\).

\begin{lem}
\label{l_:laA}
Let \([ {\bf w} ] = [w_1:\cdots:w_{n+1}] \in \mathbb{C}P^{n} \setminus \mathbb{C}P^{n-1}\) 
be a generic point as described above. Then under the family \(\{ \psi_{A_{\underline{t}}} \}\) ,
\[
A\big( \gamma( \{ \psi_{A_{\underline{t}}} \}, [{\bf w}] ) \big) =
\frac{\pi^k}{k!} \left( \frac{ \left| \frac{w_1}{w_{n+1}} \right|^2 + \cdots + \left| \frac{w_k}{w_{n+1}} \right|^2 }{ 1 + \sum_{j=1}^{n} \left| \frac{w_j}{w_{n+1}} \right|^2 } \right)^k.
\]
\end{lem}

\begin{proof}
We carry out the computation on \((B^{2n}(1), \omega_0)\), and then use the 
\(\iota\)-equivariant symplectic embedding \(j\) to translate the result into 
\((\mathbb{C}P^n, \omega_\textup{FS})\).

As discussed above, for a generic point \({\bf z} = (z_1, \ldots, z_n) \in B^{2n}(1)\), 
its trace under the family \(\{ A_{\underline{t},1} \}\) describes a \((2k-1)\)-sphere.
Define the map
\[
u_{\gamma( A_{\underline{t},1}, {\bf z} )} \colon D^{2k} \to B^{2n}(1)
\]
as the composition of a radial rescaling of the unit ball into the closed \(2k\)-ball 
of radius \(r_0 := \sqrt{|z_1|^2 + \cdots + |z_k|^2}\), followed by inclusion into 
the affine subspace  \(\{(u_1, \ldots, u_k, z_{k+1}, \ldots, z_n) \mid u_j \in 
\mathbb{C} \}\) of \(\mathbb{C}^n\). Explicitly,
\[
u_{\gamma( A_{\underline{t},1}, {\bf z} )}(v_1, \ldots, v_k) =
(r_0 v_1, \ldots, r_0 v_k, z_{k+1}, \ldots, z_n).
\]

The image of the boundary \(\partial D^{2k}\) under this map is the \((2k-1)\)-sphere 
traced out by \(\{ A_{\underline{t},1} {\bf z} \mid \underline{t} \in S^{2k-1} \}\), 
and the symplectic area enclosed by this sphere is given by
\[
A\left( \gamma( A_{\underline{t},1}, {\bf z} ) \right) =
\frac{\pi^k}{k!} \left( |z_1|^2 + \cdots + |z_k|^2 \right)^k.
\]

Now, since the symplectic embedding \(j \colon (B^{2n}(1), \omega_0) \to 
(\mathbb{C}P^n, \omega_\textup{FS})\) is \(\iota\)-equivariant, we have
\[
j\left( \gamma( A_{\underline{t},1}, {\bf z} ) \right) =
\gamma( \{ \psi_{A_{\underline{t}}} \}, [{\bf w}] ),
\]
where \([{\bf w}] = j({\bf z})\). Thus, the image of the disk \(D^{2k}\) under the 
composition \(j \circ u_{\gamma( A_{\underline{t},1}, {\bf z} )}\) is a bounding 
disk for the trace of \([{\bf w}]\) under the family \(\{ \psi_{A_{\underline{t}}} \}\) 
in \(\mathbb{C}P^n\).
Finally, from the explicit expression of \(j\) given in 
\eqref{e:simpemn}, it follows that 
\[
A( \gamma( \{ \psi_{A_{\underline{t}}} \}, [{\bf w}] ) ) =
\frac{\pi^k}{k!} \left( \frac{ \sum_{j=1}^k \left| \frac{w_j}{w_{n+1}} \right|^2 }{ 1 + 
\sum_{j=1}^n \left| \frac{w_j}{w_{n+1}} \right|^2 } \right)^k,
\]
which proves the lemma.
\end{proof}

Next, we compute the value of Weinstein's morphism on the element  
\([ \psi_{A_{\underline{t}}} ] \in \pi_{2k-1}(\textup{Ham}(\mathbb{C}P^n, 
\omega_\textup{FS}))\).  
This computation involves evaluating a specific integral, which will be addressed 
in the final section. The result is expressed as a finite sum over multi-indices, 
reflecting the structure of the integral and its symmetry properties.

\begin{prop}
\label{p:2n}
Let \(1 \leq k \leq n\), and let
\(\{ \psi_{A_{\underline{t}}} \}_{\underline{t} \in S^{2k-1}}\)
be the family of Hamiltonian diffeomorphisms of 
\((\mathbb{C}P^n, \omega_\textup{FS})\) defined above. Then, the value of Weinstein’s morphism 
on the homotopy class \([\psi_{A_{\underline{t}}}]\) is given by
 $$
\mathcal{A}([\psi_{A_{\underline{t}}}]) = \frac{\pi^k}{ k! }\cdot
\frac{n!}{(n+k)! \, 2^k}
\sum_{I} \frac{k!}{i_1! \cdots i_{2k}!}
\prod_{j=1}^{2k} (2i_j - 1)!!
 \quad \in \mathbb{R}/\langle \pi^k / k! \rangle,
$$
where the sum runs over all multi-indices \(I = (i_1, \ldots, i_{2k})\) such that \(i_j \geq 0\) and \(i_1 + \cdots + i_{2k} = k\).
\end{prop}

\begin{proof}
Fix \( 1 \leq k \leq n \).  
According to the normalization condition on the symplectic form, the volume of  
\( ( \mathbb{C}P^n, \omega_{\textup{FS}})\)  is \( \pi^n/n! \).  
Therefore, the generalized Weinstein morphism takes the form:
\[
\mathcal{A}([\psi_{A_{\underline{t}}}]) = \frac{n!}{\pi^n}
\int_{\mathbb{C}P^n} A\left( \gamma\left( \{\psi_{A_{\underline{t}}}\}, [\mathbf{w}] \right) \right)     
\frac{ \omega_{\textup{FS}}^n }{n!} \in  \mathbb{R}/\langle \pi^k / k! \rangle.
\]
Recall that the symplectic embedding \( j : (B^{2n}(1), \omega_0) \to (\mathbb{C}P^n, \omega_{\textup{FS}}) \)  
has a dense image. Consequently, we may compute the integral over the ball  
\( B^{2n}(1) \) instead of \( \mathbb{C}P^n \). Moreover, we know the value of  
\( A\left(\gamma(A_{\underline{t},1}, \mathbf{z})\right) \) for all \( \mathbf{z} \in B^{2n}(1) \),  
except on a set of measure zero.

Applying Lemma \ref{l:integralall}, which provides the explicit value of the integral,  
the result follows. For,
\begin{eqnarray*}
\mathcal{A}([\psi_{A_{\underline{t}}}])
&=&\frac{n!}{\pi^n}
\int_{  B^{2n}(1)} A (  {\gamma(  A_{{\underline{t}},1} ,  {\bf z}  )}   ) \frac{ \omega_0^n }{n!}\\
&=&\frac{n!}{\pi^{n-k}  k!}
\int_{  B^{2n}(1)}     (|z_1|^2 +\cdots +  |z_k|^2) ^k  \frac{ \omega_0^n }{n!}\\
&=&
\frac{n!}{\pi^{n-k}  k!}
\left(
\frac{\pi^{n}}{ 2^k \ (n+k)! }\sum_I    
\frac{k!}{i_1! \cdots i_{2k}!}
 (2i_1-1) !!  
  \cdots 
 (2i_{2k}-1)!!
 \right)\\
 &=&
 \frac{\pi^k \ n!}{(n+k)!\ k!\  2^k} 
\sum_I    \frac{k!}{i_1! \cdots i_{2k}!}
 (2i_1-1) !!  
  \cdots 
 (2i_{2k}-1)!!.
\end{eqnarray*}
\end{proof}

Next, we use the following lemma to simplify the expression that 
was obtained in the previous proposition.


\begin{lem}
\label{l:va}
$$
 \sum_{\substack{i_1, \dots, i_{2k} \ge 0 \\ i_1+\cdots+i_{2k}=k}} 
\frac{k!}{i_1! \cdots i_{2k}!} \,(2i_1-1)!! \cdots (2i_{2k}-1)!!= 2^k k! \binom{2k-1}{k}
$$
where the sum runs over all multi-indices \( I = (i_1, \ldots, i_{2k}) \) such that \( i_j \geq 0 \) and \( i_1 + \cdots + i_{2k} = k \).
\end{lem}
\begin{proof}
Denote by $S_k$ the sum.
 For each $i_j \ge 1$, we have
\[
(2i_j-1)!! = \frac{(2i_j)!}{2^{i_j} i_j!}, 
\]
 Then each summand becomes
\begin{eqnarray*}
\frac{k!}{i_1!\cdots i_{2k}!} \prod_{j=1}^{2k} (2i_j-1)!! 
&=& k! \prod_{j=1}^{2k} \frac{(2i_j)!}{2^{i_j} (i_j!)^2} \\
&=& \frac{k!}{2^k} \prod_{j=1}^{2k} \binom{2i_j}{i_j}
\end{eqnarray*}
since     $i_1+\cdots+i_{2k}=k$. 
Hence,
\[
S_k = \frac{k!}{2^k} \sum_{I} 
\prod_{j=1}^{2k} \binom{2i_j}{i_j}.
\]
Using the generating function
\[
G(x) = \sum_{m\ge 0} \binom{2m}{m} x^m = \frac{1}{\sqrt{1-4x}}
\]
we get that 
\begin{eqnarray*}
\sum_{i_1+\cdots+i_{2k}=k} \prod_{j=1}^{2k} \binom{2i_j}{i_j} 
&=& [x^k] \big(G(x)\big)^{2k} = [x^k] (1-4x)^{-k}\\
&=&4^k \binom{2k-1}{k}
\end{eqnarray*}
and therefore,
\[
S_k = \frac{k!}{2^k} \cdot 4^k \binom{2k-1}{k} = 2^k k! \binom{2k-1}{k}.
\]

\end{proof}


We end this section with the proof that
$ \pi_{2k-1} ( \textup{Ham}(\mathbb{C}P^n,\omega_\textup{FS}) )$ 
is non trivial
for any $1\leq k\leq n$.

\begin{proof}[Proof of Theorem \ref{t:main}]
Let $[\psi_{A_{\underline{t}}}] \in \pi_{2k-1} ( \textup{Ham}(\mathbb{C}P^n,\omega_\textup{FS}) ) $
be the element of Proposition \ref{p:2n}. Then by  Lemma \ref{l:va}, its value under the Weinstein's morphism
takes the form
\begin{eqnarray*}
\mathcal{A}([\psi_{A_{\underline{t}}}])
=\frac{\pi^k}{k!} \cdot
 \frac{ n! \  k!}{(n+k)!\  } 
 \binom{2k-1}{k}  \in \mathbb{R}/\langle \pi^k/k!\rangle
\end{eqnarray*}

Since $k\leq n,$
the factor of $\pi^k/k!$ is a rational number less than 1. Therefore
$[\psi_{A_{\underline{t}}}]\neq 0  $ and $ \pi_{2k-1} (  \textup{Ham}
(\mathbb{C}P^n,\omega_\textup{FS}) )$ is nontrivial.
\end{proof}

Now that we have computed 
$\mathcal{A}$ on
$[\psi_{A_{\underline{t}}}]  \in \pi_{2k-1}\!\left(\mathrm{Ham}
 (\mathbb{C}P^n, \omega_{\mathrm{FS}})\right)$,
we are in a position to evaluate Weinstein's morphism
on groups of the form
$$
\pi_{2k-1}\!\left(\mathrm{Ham}(\mathbb{C}P^n\times M, \, 
\omega_{\mathrm{FS}}\oplus \omega)\right).
$$

\medskip

\begin{proof}[Proof of Theorem~\ref{t:prod}]
Let 
$
[\psi_{\underline{t}}] \in \pi_{2k-1}\!\left(\mathrm{Ham}(\mathbb{C}P^n, \omega_{\mathrm{FS}})\right)
$
be the nontrivial element given in Proposition~\ref{p:2n} and $(M,\omega)$ a closed symplectic manifold
such that $\pi_{2k-1}(M)$ is trivial.

It follows by Proposition \ref{p:product}, that
\begin{eqnarray*}
\mathcal{A}^{\mathbb{C}P^n \times M }[\psi_{\underline{t}} \times 1_M]
&=&\big[\,    
\mathcal{A}^{\mathbb{C}P^n  }[\psi_{\underline{t}}]
   \,\big]  \\
   &=&\left[\,     \frac{\pi^k}{k!}  \cdot q \,\right] 
   \in \mathcal{P}_{2k}(\mathbb{C}P^n\times M, \, 
\omega_{\mathrm{FS}}\oplus \omega).
\end{eqnarray*}
for some $q\in \mathbb{Q}$ such that $0<q<1.$

Note that
$\mathcal{P}_{2k}(\mathbb{C}P^n\times M, \, 
\omega_{\mathrm{FS}}\oplus \omega)$ 
is  generated by
\begin{eqnarray*}
\left\{
\frac{\pi^k}{k!},\;
\frac{\pi^{k-1}}{(k-1)!}\cdot 
\mathcal{P}_{2}(M, \omega),\;\ldots,\;
\pi\cdot \mathcal{P}_{2k-1}(M, \omega),\;
\mathcal{P}_{2k}(M, \omega)
\right\}.
\end{eqnarray*}
Since $(M,\omega)$ is such that all $\mathcal{P}_{2j}(M,\omega)$ are contained in $\mathbb{Q}$,
it follows  that $[\psi_{\underline{t}} \times 1_M]$ is nontrivial in 
$ \pi_{2k-1}\!\left(\mathrm{Ham}(\mathbb{C}P^n\times M, \, 
\omega_{\mathrm{FS}}\oplus \omega)\right).$
\end{proof}

\section{Application: $\pi_{2k-1}(\textup{Ham}  (\widetilde{\mathbb{C}P}\,^n,\widetilde\omega_\rho)   )$  for
$1\leq k\leq n$}
\label{hamblowdiff}

In this section, we prove the main result of the paper. Namely, for any 
\(1 \leq k \leq n\), the homotopy group
$
\pi_{2k-1}(\textup{Ham}(\widetilde{\mathbb{C}P}\,^n, \widetilde{\omega}_\rho))
$
contains an infinite cyclic subgroup, which arises from a family of
 Hamiltonian diffeomorphisms on \((\mathbb{C}P^n, \omega_{\textup{FS}})\). 
This is achieved by lifting the previously constructed families to the 
symplectic blow-up \((\widetilde{\mathbb{C}P}\,^n, \widetilde{\omega}_\rho)\), 
and by analyzing their behavior under the generalized Weinstein morphism.

Recall that the symplectic one-point blow-up of weight \(\rho\) is constructed by 
symplectically embedding the closed ball \((\overline{B^{2n}}(\rho), \omega_0)\) 
into \((M, \omega)\), and collapsing the image of its boundary via the Hopf fibration. 
Importantly, this symplectic embedding extends to a neighborhood of the closed ball.
Fix a symplectic embedding 
$
\iota\colon  (\overline{B^{2n}}(\rho), \omega_0) \to (M, \omega).
$
We say that a Hamiltonian diffeomorphism \(\psi\) is {\em $\iota$-unitary} if,
\(\psi  (   \iota (\overline{B^{2n}}(\rho)) ) =\iota (\overline{B^{2n}}(\rho))    \)
and on a 
neighborhood of \(\overline{B^{2n}}(\rho)\), the map \(\iota^{-1} \circ \psi \circ \iota\) 
is given by the action of a unitary matrix.
In this case, \(\psi\) induces a well-defined Hamiltonian diffeomorphism 
\(\widetilde{\psi}\) on the symplectic blow-up \((\widetilde{M}, 
\widetilde{\omega}_\rho)\). Note that an $\iota$-unitary Hamiltonian
fixes the point \(\iota(0)\). This class of Hamiltonian diffeomorphisms was previously 
studied by the author in \cite{pea-rankham}.

Next, suppose that \([ \{\psi_{\underline{t}}\} ]\) is a class in \(\pi_{2k-1}(
\textup{Ham}(M, \omega))\) such that each \(\psi_{\underline{t}}\) is 
$\iota$-unitary. Then, for any point \(p \in \iota(\overline{B^{2n}}(\rho))\), 
its trace 
$ \gamma( \{\psi_{\underline{t}}\} _{{\underline{t}}\in S^{2k-1}}   ,p   )$,
is entirely contained in \(\iota(\overline{B^{2n}}(\rho))\). That is, the entire image 
of the trace lies within the image of the embedded ball.

Thus, fix \(k \in \{1, \ldots, n\}\) and assume that \(\pi_{2k-1}(M) = 0\). It follows 
that \(\pi_{2k-1}(\widetilde{M}) = 0\) as well, and hence the generalized 
Weinstein morphism is defined on both groups:
$
\pi_{2k-1}(\textup{Ham}(M,\omega))$ and $ \pi_{2k-1}(\textup{Ham}(\widetilde{M},
 \widetilde{\omega}_\rho)).
$
In this setting, the relation between the period groups of the symplectic 
manifolds is given by:
\[
\mathcal{P}_{2k}(\widetilde{M}, \widetilde{\omega}_\rho) = \mathcal{P}_{2k}
(M, \omega) + \mathbb{Z} \left\langle {\pi^k \rho^{2k}}/{k!} \right\rangle.
\]
Let \([ \{\psi_{\underline{t}}\} ] \in \pi_{2k-1}(\textup{Ham}(M, \omega))\) 
be such that each \(\psi_{\underline{t}}\) is \(\iota\)-unitary. Then the 
induced family \(\{ \widetilde{\psi}_{\underline{t}} \}\) defines a class
$
[ \{\widetilde{\psi}_{\underline{t}} \} ] \in \pi_{2k-1}(\textup{Ham}
(\widetilde{M}, \widetilde{\omega}_\rho)).
$
We now proceed to compute \(\mathcal{A}([ \widetilde{\psi}_{\underline{t}} ])\) 
in terms of the original manifold \((M, \omega)\) and the class
\([ \{\psi_{\underline{t}}\} ]\).  
This result is analogous to the one established in \cite{pea-rankham} for the 
case \(k = 1\), where the computation is carried out using the Hamiltonian 
function directly.

\begin{prop}
\label{p:relationW}
Let \([ \{\psi_{\underline{t}} \} ] \in \pi_{2k-1}(\textup{Ham}(M, \omega))\) be a class such that each \(\psi_{\underline{t}}\) is \(\iota\)-unitary, and let \([ \{\widetilde{\psi}_{\underline{t}} \} ] \in \pi_{2k-1}(\textup{Ham}(\widetilde{M}, \widetilde{\omega}_\rho))\) be the induced class on the blow-up. Then,
\begin{eqnarray*}
\mathcal{A}[\widetilde \psi_{\underline t}]   &=&\left[ 
\frac{1}{\textup{Vol}(\widetilde M,\widetilde\omega_\rho^n/n!)} 
\left(\ \  \int_{M} 
A (    \gamma( \{\psi_{\underline{t}}\} _{{\underline{t}}\in S^{2k-1}}   ,p   )  )
 \frac{ \omega^n }{n!}  - \right.\right. \\
& &\left.\left.   
 \int_{\iota(B^{2n}(\rho))} 
A (    \gamma( \{\psi_{\underline{t}}\} _{{\underline{t}}\in S^{2k-1}}   ,p   )  )
 \frac{ \omega^n }{n!}\ \
\right)\right]   \in \mathbb{R}/
\mathcal{P}_{2k}
(\widetilde M,\widetilde\omega_\rho).
\end{eqnarray*}
\end{prop}
\begin{proof}
Each Hamiltonian \(\psi_{\underline{t}}\) is \(\iota\)-unitary, and thus it maps the 
ball \(\iota(\overline{B^{2n}}(\rho))\) to itself. Therefore, for any point 
\(p \in \iota(\overline{B^{2n}}(\rho))\), there exists a map \(D^{2k} \to M\) whose 
boundary is mapped to the trace  
$
\gamma(\{\psi_{\underline{t}}\}_{\underline{t} \in S^{2k-1}}, p),
$
and whose image lies entirely within \(\iota(\overline{B^{2n}}(\rho))\).
Likewise, for any point \(p \in M \setminus \iota(\overline{B^{2n}}(\rho))\), there 
exists a map \(D^{2k} \to M\) that maps the boundary to the same type of trace 
and whose image avoids the region \(\iota(\overline{B^{2n}}(\rho))\). In this case, 
the map \(D^{2k} \to M\) lifts to a map \(D^{2k} \to \widetilde{M}\), whose boundary maps to  
$
\gamma(\{\widetilde{\psi}_{\underline{t}}\}_{\underline{t} \in S^{2k-1}}, p).
$

Henceforth, by avoiding the points in the exceptional divisor 
 $E\subset \widetilde M,$ we have
\begin{eqnarray*}
\int_{\widetilde M} 
A (    \gamma( \{\widetilde \psi_{\underline{t}}\} _{{\underline{t}}\in S^{2k-1}}   ,p   )  )
 \frac{ \widetilde\omega^n_\rho }{n!}
&=&
\int_{\widetilde M\setminus E} 
A (    \gamma( \{\widetilde \psi_{\underline{t}}\} _{{\underline{t}}\in S^{2k-1}}   ,p   )  )
 \frac{ \widetilde\omega^n_\rho }{n!}\\
&=&
\int_{M\setminus \iota (\overline{B^{2n}}(\rho))} 
A (    \gamma( \{\psi_{\underline{t}}\} _{{\underline{t}}\in S^{2k-1}}   ,p   )  )
 \frac{ \omega^n }{n!},
\end{eqnarray*}
and the claim follows.
\end{proof}

Once again, consider \(({\mathbb{C}P}^n, \omega_\textup{FS})\), where the 
symplectic form is normalized as before, and assume \(n > 1\). Fix \(k \in \{1, 
\ldots, n\}\), and consider the element in  
\(\pi_{2k-1}(\textup{Ham}({\mathbb{C}P}^n, \omega_\textup{FS}))\) represented
 by the family of Hamiltonian diffeomorphisms \(\{\psi_{\underline{t}}\}\) induced 
by the unitary matrices described in Lemma \ref{l:matrix}.

Fix \(0 < \rho < 1\), and restrict the symplectic embedding  
\(j \colon (B^{2n}(1), \omega_0) \to (\mathbb{C}P^n, \omega_\textup{FS})\),
 defined in \eqref{e:simpemn}, to the closed ball \(\overline{B^{2n}}(\rho)\). 
This restriction gives rise to the one-point symplectic blow-up  
\((\widetilde{\mathbb{C}P}\,^n, \widetilde{\omega}_\rho)\) of weight \(\rho\).  
By construction, the Hamiltonian diffeomorphisms \(\psi_{\underline{t}}\) are 
\(j\)-unitary, and thus induce a class  
\([ \widetilde{\psi}_{\underline{t}} ] \in \pi_{2k-1}(\textup{Ham}(
\widetilde{\mathbb{C}P}\,^n, \widetilde{\omega}_\rho))\).
We proceed to compute \(\mathcal{A}([\widetilde{\psi}_{\underline{t}}])\) 
using the formula established in Proposition \ref{p:relationW}. Finally, we will 
show that this class has infinite order, thereby proving the main result of the paper.

\begin{prop}
\label{p:ainblow}
Let \(1 \leq k \leq n\), and let
\(\{   \widetilde \psi_{\underline t}   \}_{\underline{t} \in S^{2k-1}}\)
be the family of Hamiltonian diffeomorphisms of  
$(\widetilde{\mathbb{C}P}\,^n, \widetilde{\omega}_\rho)$
 defined above. Then, the value 
of Weinstein’s morphism on the homotopy class \([\widetilde \psi_{\underline t}] \) is given by

\begin{eqnarray*}
\mathcal{A}[\widetilde \psi_{\underline t}]   &=&\left[ 
\frac{n!\pi^k}{( 1-\rho^{2n}) k!}
\cdot
 \frac{1- \ \rho^{2(n+k)}   }{ 2^k \ (n+k)! }
\cdot
\sum_I    
\frac{k!}{i_1!\cdots i_{2k}!}  
\prod_{j=1}^{2k} (2i_j - 1)!!
\right]
\end{eqnarray*}
in $\mathbb{R}/\langle \pi^k/k!, \pi^k\rho^{2k} /k!\rangle$.
\end{prop}

\begin{proof}
From the definition of the symplectic one-point blow-up, we get that
$
\textup{Vol}(\widetilde{\mathbb{C}P}\,^n,\widetilde\omega_\rho^n/n!)$ is equal to  ${\pi^n(1 - \rho^{2n})}/{n!}.
$

Next, we apply Lemma \ref{l:integralall} twice to evaluate the integrals involved in Proposition \ref{p:relationW}.
First, we compute the full integral over \(\mathbb{C}P^n\):
\begin{align*}
\int_{\mathbb{C}P^n} 
A\left( \gamma( \{\psi_{\underline{t}}\} _{{\underline{t}}\in S^{2k-1}}, p ) \right) 
\frac{ \omega_\textup{FS}^n }{n!}
&= \int_{B^{2n}(1)} 
A\left( \gamma( A_{{\underline t}, 1}, {\bf z} ) \right) 
\frac{ \omega_0^n }{n!} \\
&= \frac{\pi^n}{2^k (n+k)!} \cdot \frac{\pi^k}{k!} \sum_I    
\frac{k!}{i_1!\cdots i_{2k}!}  
\prod_{j=1}^{2k} (2i_j - 1)!!.
\end{align*}

Similarly, the integral over the image of the embedded ball becomes:
\begin{align*}
\int_{\iota(B^{2n}(\rho))} 
A\left( \gamma( \{\psi_{\underline{t}}\} _{{\underline{t}}\in S^{2k-1}}, p ) \right) 
\frac{ \omega_\textup{FS}^n }{n!}
&= \int_{ B^{2n}(\rho) } 
A\left( \gamma( A_{{\underline t}, 1}, {\bf z} ) \right) 
\frac{ \omega_0^n }{n!} \\
&= \frac{\pi^n \rho^{2(n+k)}}{2^k (n+k)!} \cdot \frac{\pi^k}{k!} 
\sum_I    
\frac{k!}{i_1!\cdots i_{2k}!}  
\prod_{j=1}^{2k} (2i_j - 1)!!.
\end{align*}

Putting these expressions together in Proposition \ref{p:relationW}, we get
\begin{align*}
\mathcal{A}[\widetilde \psi_{\underline t}] 
&= \left[
\frac{n! \pi^k}{(1 - \rho^{2n}) k!} \cdot 
\frac{1 - \rho^{2(n+k)}}{2^k (n+k)!} \cdot 
\sum_I    
\frac{k!}{i_1!\cdots i_{2k}!}  
\prod_{j=1}^{2k} (2i_j - 1)!!
\right]
\end{align*}
in \(\mathbb{R}/\langle \pi^k/k!, \pi^k \rho^{2k}/k! \rangle\).
\end{proof}

\begin{proof}[Proof of Theorem \ref{t:realmain}]
From Proposition \ref{p:ainblow}, it suffices to show that the quantity
\[
\frac{n!}{1 - \rho^{2n}} 
\cdot
\frac{1 - \rho^{2(n+k)}}{2^k (n+k)!}
\cdot
\sum_I    
\frac{k!}{i_1! \cdots i_{2k}!}  
\prod_{j=1}^{2k} (2i_j - 1)!!
\]
is not of the form \(A + B \rho^{2k}\) for some integers \(A, B \in \mathbb{Z}\).
This would amount to expressing the quantity as a polynomial in \(\rho\) with rational 
coefficients and setting it equal to zero. 

However, since \(\rho\) is transcendental, such a relation cannot hold. Therefore, we conclude that
$
[\widetilde \psi_{\underline t}] \in 
\pi_{2k-1}(\textup{Ham}(\widetilde{\mathbb{C}P}\,^n, \widetilde{\omega}_\rho))
$
has infinite order.
\end{proof}

\section{Appendix}

For completeness, we provide a brief sketch of the calculation of the following integral,
which played a central role in the previous sections in the evaluation of Weinstein’s morphism.

\begin{lem}
\label{l:integralall}
Let \( 1 \leq l \leq n \), \( k \) be a positive integer, and \( r_0 > 0 \). We have the following identity:
\begin{eqnarray}
\label{e:integralK}
\int_{B^{2n}(r_0)}    ( |z_1|^2 \!\!\!\! & + &\!\!\!\!  \cdots  +|z_l|^2)^k  \ \ \frac{\omega_0^n}{n!} =\\
\nonumber & &
 \frac{\pi^{n} \ r_0^{2(n+k)}   }{ 2^k \ (n+k)! }\sum_I    
\frac{k!}{i_1!\cdots i_{2l}!}  
\prod_{j=1}^{2l} (2i_j - 1)!!
\end{eqnarray}
where the sum runs over all multi-indices \( I = (i_1, \ldots, i_{2l}) \) such that \( i_j \geq 0 \) and \( i_1 + \cdots + i_{2l} = k \).
\end{lem}

We compute the integral using the spherical coordinate system in 
\( \mathbb{R}^{2n} \). To this end, let \( (r, \theta_1, \ldots, \theta_{2n-1}) \) denote the
 spherical coordinate system, where \( r \) is the radial distance and
  \( \theta_1, \theta_2, \ldots, \theta_{2n-1} \) are the angular coordinates. 
The relations between the spherical and Cartesian coordinates are given by:
\begin{eqnarray*}
x_1&=&r\cos\theta_1\\
x_2&=&r\sin\theta_1 \cos\theta_2\\
....& & ....\\
x_{2n-2}&=&r\sin\theta_1 \sin\theta_2    \cdots \sin\theta_{2n-3}\cos\theta_{2n-2}\\
x_{2n-1}&=&r\sin\theta_1 \sin\theta_2    \cdots \sin\theta_{2n-3}\sin\theta_{2n-2} \cos\theta_{2n-1}\\
x_{2n}&=&r\sin\theta_1 \sin\theta_2    \cdots \sin\theta_{2n-3}\sin\theta_{2n-2} \sin\theta_{2n-1}.
\end{eqnarray*}

To compute the integral in (\ref{e:integralK}), we consider two cases: 
 $l<n$ and $l=n$.
 In the case when $l<n$, we have the following setup for the integral,
\begin{eqnarray*}
(x_1^2+x_2^2+&\cdots& +x_{2k-1}^2+x_{2l}^2)^k =
\sum_I  
\frac{k!}{i_1! \cdots i_{2l}!}
x_1^{2i_1}   x_2^{2i_2} \cdots x_{2l-1}^{2i_{2l-1}} x^{2i_{2l}}_{2l}\\
&=& r^{2k}\sum_I   
\frac{k!}{i_1! \cdots i_{2l}!}
 (\cos^{2i_1}\theta_1 ) (\sin^{2i_2}\theta_1 \cos^{2i_2}\theta_2) \\
&&\cdots (   \sin^{2i_{2l}}\theta_1 \sin^{2i_{2l}}\theta_2    \cdots \sin^{2i_{2l}}\theta_{2l-1}\cos^{2i_{2l}}\theta_{2l}  )\\
&=& r^{2k}\sum_I 
\frac{k!}{i_1! \cdots i_{2l}!}
(\cos^{2i_1}\theta_1     \sin^{2(i_2+\cdots+i_{2l})}\theta_1           )\\
&&(\cos^{2i_2}\theta_2     \sin^{2(i_3+\cdots+i_{2l})}\theta_2           )
 \cdots
(\cos^{2i_{2l-1}}\theta_{2l-1}     \sin^{2i_{2l}}\theta_{2l-1}           )\\
&&
(\cos^{2i_{2l}}\theta_{2l}              )
\end{eqnarray*}
where the sum is as indicated in the lemma. To abbreviate notation, we define 
$a_I$  as $\frac{k!}{i_1! \cdots i_{2l}!}$,
where \( I = (i_1, \ldots, i_{2l}) \) is the multi-index corresponding to the powers of the coordinates in the integrand. 
The Jacobian determinant of the change of coordinate system, when converting from Cartesian coordinates 
to spherical coordinates, only involves the first \( 2n-1 \) variables
and is given by
\begin{eqnarray}
\label{e:det}
r^{2n-1}
\sin^{2n-2}\theta_1 \sin^{2n-3}\theta_2    \cdots 
\sin^{3}\theta_{2n-4}\sin^{2}\theta_{2n-3} \sin\theta_{2n-2}.
\end{eqnarray}
Thus, after multiplying \( (x_1^2 + \cdots + x_{2l}^2)^k \) by the expression above, we obtain
\begin{eqnarray*}
r^{2(n+k)-1} \cdot
\sin^{2n-l-2}\theta_{2l+1} \sin^{2n-2l-3}\theta_{2l+2}  \cdots 
\sin^{2}\theta_{2n-3} \sin\theta_{2n-2} \cdot \\
\sum_I 
a_I
(\cos^{2i_1}\theta_1     \sin^{2(i_2+\cdots+i_{2l}) +2n-2    }\theta_1           )
(\cos^{2i_2}\theta_2     \sin^{2(i_3+\cdots+i_{2l})+2n-3}\theta_2           )\\
\cdots
(\cos^{2i_{2l-1}}\theta_{2l-1}     \sin^{2i_{2k}   +2n-2l}\theta_{2l-1}           )
(\cos^{2i_{2l}}\theta_{2l}    \sin^{2n-2l-1}\theta_{2l}          )
\end{eqnarray*}

Next, we proceed to compute the integral over \( (0, r_0] \times [0, \pi)^{2n-2} \times [0, 2\pi) \), 
and express the final result in terms of the Beta function.
\begin{eqnarray*}
&& \frac{r_0^{2(n+k)}  }{2(n+k)}\cdot 2\pi\cdot B\left(\frac{ 1}{2}, \frac{ 2n-2l -1 }{2}\right)
B\left(\frac{ 1 }{2}, \frac{ 2n-2l-2  }{2}\right)\cdots\\
&&
B\left(\frac{ 1}{2}, \frac{2+1 }{2}\right)
B\left(\frac{1 }{2}, \frac{ 1+1 }{2}\right)
 \sum_I  a_I
B\left(\frac{ 2i_1+1 }{2}, \frac{ 2(i_2+\cdots +i_{2l}) +2n-1  }{2}\right)\cdot \\
&&B\left(\frac{ 2i_2+1 }{2}, \frac{ 2(i_3+\cdots +i_{2l}) +2n-2  }{2}\right)\cdots\\
&&
B\left(\frac{ 2i_{2l-1}+1 }{2}, \frac{ 2i_{2l} +2n-2l+1  }{2}\right)
B\left(\frac{ 2i_{2l}+1 }{2}, \frac{ 2n-2l  }{2}\right).
\end{eqnarray*}
Note that there are \( 2l \) terms in each summand. Using the Gamma function, 
the above expression takes the following form:
\begin{eqnarray*}
&&\frac{\pi \ r_0^{2(n+k)}  }{n+k}
\frac{
\Gamma\left(
\frac{1 }{2} 
\right)  
\Gamma\left( 
\frac{ 2n-2l-1}{2}
\right)
}
{\Gamma\left( \frac{ 2n -2l }{2}\right) }
\frac{
\Gamma\left(
\frac{1 }{2} 
\right)  
\Gamma\left( 
\frac{ 2n-2l-2}{2}
\right)
}
{\Gamma\left( \frac{ 2n -2l-1 }{2}\right) }\cdots \\
&&
\frac{
\Gamma\left(
\frac{1 }{2} 
\right)  
\Gamma\left( 
\frac{ 2+1}{2}
\right)
}
{\Gamma\left( \frac{ 4 }{2}\right) }
\frac{
\Gamma\left(
\frac{1 }{2} 
\right)  
\Gamma\left( 
\frac{ 1+1}{2}
\right)
}
{\Gamma\left( \frac{ 3 }{2}\right) }
\sum_I a_I 
\frac{
\Gamma\left(
\frac{ 2i_1+1 }{2} 
\right)  
\Gamma\left( 
\frac{ 2(i_2+\cdots +i_{2l}) +2n-1}{2}
\right)
}
{\Gamma\left( \frac{ 2(i_1+\cdots +i_{2l}) +2n  }{2}\right) }\\
&&
\frac{
\Gamma\left(
\frac{ 2i_2+1 }{2} 
\right)  
\Gamma\left( 
\frac{ 2(i_3+\cdots +i_{2l}) +2n-2}{2}
\right)
}
{\Gamma\left( \frac{ 2(i_{2l-1}+1) +2n -1 }{2}\right) }\cdots\\
&&
\frac{
\Gamma\left(
\frac{ 2i_{2l-1}+1 }{2} 
\right)  
\Gamma\left( 
\frac{ 2i_{2l} +2n-2l+1}{2}
\right)
}
{\Gamma\left( \frac{ 2(i_{2l-1} +i_{2l}) +2n -2l+2 }{2}\right) }
\frac{
\Gamma\left(
\frac{ 2i_{2l}+1 }{2} 
\right)  
\Gamma\left( 
\frac{ 2n-2l}{2}
\right)
}
{\Gamma\left( \frac{ 2i_{2l} +2n -2l+1 }{2}\right) }.
\end{eqnarray*}
After canceling the matching terms, the expression simplifies to:
\begin{eqnarray*}
&&\frac{\pi \ r_0^{2(n+k)}  }{n+k}   
\Gamma\left(
\frac{ 1 }{2} 
\right)  ^{2n-2l-2}
\Gamma(1)
\sum_I a_I  
\frac{
\Gamma\left(
\frac{ 2i_1+1 }{2} 
\right)    
\cdots
\Gamma\left(
\frac{ 2i_{2l}+1 }{2} 
\right)  
}
{\Gamma\left( \frac{ 2(i_1+\cdots +i_{2l}) +2n  }{2}\right) }
\\
&=&
\frac{\pi^{n-l}\ r_0^{2(n+k)}    }{(n+k)!}\sum_I   a_I 
\Gamma\left(
\frac{ 2i_1+1 }{2} 
\right)  \cdots \Gamma\left(
\frac{ 2i_{2l}+1 }{2} 
\right)\\
&=&
\frac{\pi^{n} \ r_0^{2(n+k)}  }{(n+k)!}\sum_I 
\frac{k!}{i_1!\cdots i_{2l}!}   
\frac{ 1\cdot 3\cdots (2i_1-1) }{2^{i_1}} 
  \cdots 
\frac{ 1\cdot 3\cdots (2i_{2l}-1) }{2^{i_{2l}}}. 
\end{eqnarray*}
In the case when \( i_j = 0 \), we define $(2\cdot 0-1)/2^{0}$   equal to 1.
Henceforth, for \( 1 \leq l < n \) and \( k \) a positive integer, the value of the integral (\ref{e:integralK}) is:
\begin{eqnarray*}
\int_{B^{2n}(r_0)}    ( |z_1|^2+ &\cdots&  +|z_l|^2)^k  \ \ \frac{\omega_0^n}{n!} =\\
& &\frac{\pi^{n} \ r_0^{2(n+k)}   }{ 2^k \ (n+k)! }\sum_I    
\frac{k!}{i_1!\cdots i_{2l}!}  
(2i_1-1) !!  
  \cdots 
 (2i_{2l}-1)!!. 
\end{eqnarray*}

\medskip
The case when $l=n$ follows the same approach as the computation above. For
\begin{eqnarray*}
(x_1^2+&\cdots& +x_{2n}^2)^k\\
&=& r^{2k}\sum_I  a_I
(\cos^{2i_1}\theta_1     \sin^{2(i_2+\cdots+i_{2n})}\theta_1           )
(\cos^{2i_2}\theta_2     \sin^{2(i_3+\cdots+i_{2n})}\theta_2           )\cdots\\
&& 
(\cos^{2i_{2n-2}}\theta_{2n-2}     \sin^{2(i_{2n-1}+i_{2n})}\theta_{2n-2})
(\cos^{2i_{2n-1}}\theta_{2n-1}     \sin^{2i_{2n}}\theta_{2n-1}).
\end{eqnarray*}
Thus, after multiplying by the Jacobian determinant (\ref{e:det}), we obtain:
\begin{eqnarray*}
r^{2(n+k)-1}\sum_I  a_I
(\cos^{2i_1}\theta_1     \sin^{2(i_2+\cdots+i_{2n}) +2n-2}\theta_1           )
(\cos^{2i_2}\theta_2     \sin^{2(i_3+\cdots+i_{2n})+2n-3}\theta_2           )\cdots\\
(\cos^{2i_{2n-2}}\theta_{2n-2}     \sin^{2(i_{2n-1}+i_{2n})+1}\theta_{2n-2})
(\cos^{2i_{2n-1}}\theta_{2n-1}     \sin^{2i_{2n}}\theta_{2n-1}).
\end{eqnarray*}

As before, we compute the integral of the above expression over $(0,r_0]\times [0,\pi)^{2n-2}\times [0,2\pi)$),
 and express the result in terms of the Beta function:
\begin{eqnarray*}
&& \frac{r_0^{2(n+k)}  }{2(n+k)}\sum_I  a_I
B\left(\frac{ 2i_1+1 }{2}, \frac{ 2(i_2+\cdots +i_{2n}) +2n-1  }{2}\right)\\
&&B\left(\frac{ 2i_2+1 }{2}, \frac{ 2(i_3+\cdots +i_{2n}) +2n-2  }{2}\right)\cdots\\
&&
B\left(\frac{ 2i_{2n-2}+1 }{2}, \frac{ 2(i_{2n-1} + i_{2n})+2  }{2}\right)
2B\left(\frac{ 2i_{2n-1}+1 }{2}, \frac{2i_{2n}+1  }{2}\right)
\end{eqnarray*}
This time, there are $2n-1$
terms in each summand. Using the Gamma function as before, we obtain:
\begin{eqnarray*}
&&\frac{r_0^{2(n+k)} }{n+k}\sum_I  a_I
\frac{
\Gamma\left(
\frac{ 2i_1+1 }{2} 
\right)  
\Gamma\left( 
\frac{ 2(i_2+\cdots +i_{2n}) +2n-1}{2}
\right)
}
{\Gamma\left( \frac{ 2(i_1+\cdots +i_{2n}) +2n  }{2}\right)}
\frac{
\Gamma\left(
\frac{ 2i_2+1 }{2} 
\right)  
\Gamma\left( 
\frac{ 2(i_3+\cdots +i_{2n}) +2n-2}{2}
\right)
}
{\Gamma\left( \frac{ 2(i_{2}+\cdots+ i_{2n}) +2n -1 }{2}\right) }\\
&&\cdots 
\frac{
\Gamma\left(
\frac{ 2i_{2n-2}+1 }{2} 
\right)  
\Gamma\left( 
\frac{ 2(i_{2n-1}+i_{2n}) +2}{2}
\right)
}
{\Gamma\left( \frac{ 2(i_{2n-2}+i_{2n-1} + i_{2n}) +3}{2}\right) }
\frac{
\Gamma\left(
\frac{ 2i_{2n-1}+1 }{2} 
\right)  
\Gamma\left( 
\frac{ 2i_{2n} +1}{2}
\right)
}
{\Gamma\left( \frac{ 2(i_{2n-1}+i_{2n}) +2}{2}\right) }.
\end{eqnarray*}

Finally, this expression simplifies to yield a result similar to the case 
$k<n$,
\begin{eqnarray*}
\frac{r_0^{2(n+k)}}{n+k}\sum_I  &a_I&
\frac{
\Gamma\left(
\frac{ 2i_1+1 }{2} 
\right)  
\cdots
\Gamma\left(
\frac{ 2i_{2n}+1 }{2} 
\right)  
}
{\Gamma\left( \frac{ 2(i_1+\cdots +i_{2n}) +2n  }{2}\right)}\\
&=&\frac{r_0^{2(n+k)}}{(n+k)!}\sum_I a_I
\Gamma\left(
\frac{ 2i_1+1 }{2} 
\right)  
\cdots 
\Gamma\left(
\frac{ 2i_{2n}+1 }{2} 
\right)\\
&=&
\frac{\pi^n\ r_0^{2(n+k)} }{2^{n}\ (n+k)!}
\sum_I    
\frac{k!}{i_1!\cdots i_{2n}!}
 (2i_1-1) !!  
  \cdots 
 (2i_{2n}-1)!!. 
\end{eqnarray*}

\bibliographystyle{acm}

\bibliography{/Users/andrespedroza/Dropbox/Documentostex/Ref.bib} 
\end{document}